\documentclass[a4paper,11pt]{amsart}

\usepackage[left=2.5cm,right=2.5cm]{geometry}
\usepackage[english]{babel}
\usepackage[utf8]{inputenc}
\usepackage{amsmath, amsfonts, amssymb,amsthm}

\newtheorem{Theorem}{Theorem}[section]
\newtheorem{Lemma}[Theorem]{Lemma}

\newtheorem{Prop}[Theorem]{Proposition}

\theoremstyle{definition}
\newtheorem{Rmk}[Theorem]{Remark}
\newtheorem{Ex}[Theorem]{Example}		

\newcommand{\Fp}{\mathbb{F}_p}
\newcommand{\Fq}[1][]{\mathbb{F}_{q^{#1}}}

\newcommand{\Z}{\mathbb{Z}}
\def\ie{\textit{i.e.}}

\DeclareMathOperator{\End}{End}
\DeclareMathOperator{\Endo}{End^{0}}

\title[Cyclicity of abelian varieties over finite fields]{On the cyclicity of the rational points group of abelian varieties over finite fields}
\author{Alejandro J. Giangreco-Maidana}
\address{Aix Marseille Universit\'e, CNRS, Centrale Marseille, I2M UMR 7373, 13453 Marseille, France}
\email{ajgiangreco@gmail.com}

\begin{document}


\begin{abstract}
We propose a simple criterion to know if  an abelian variety $A$ defined over a finite field $\Fq$ is \emph{cyclic}, \emph{i.e.}, it has a cyclic group of rational points; this criterion is based on the endomorphism ring $\End_{\Fq}(A)$. We also provide a criterion to know if an isogeny class is \emph{cyclic}, \ie, all its varieties are cyclic; this criterion is based on the characteristic polynomial of the isogeny class. We find some asymptotic lower bounds on the fraction of cyclic $\Fq$-isogeny classes among certain families of them, when $q$ tends to infinity. Some of these bounds require an additional hypothesis. In the case of surfaces, we prove that this hypothesis is achieved and, over all $\Fq$-isogeny classes with endomorphism algebra being a field and where $q$ is an even power of a prime, we prove that the one with maximal number of rational points is cyclic and ordinary. 
\end{abstract}
\maketitle
\section{Introduction}
The group structure of elliptic curves (EC) defined over finite fields 
has a theoretical interest, but also an interest in applications to cryptography and error-correcting codes. More precisely, those elliptic curves with cyclic groups of rational points are of special interest for various problems (see for example \cite{koblitz1987elliptic}, \cite{Morain1991} and \cite{Kaliski1991}).
In many algorithms that use elliptic curves over finite fields one chooses it at random, thus, knowing such statistics can be useful.

The structure of all possible groups for EC
over a finite field was discovered (independently)
in \cite{SCHOOF1987183}, \cite{ruck1987note},\cite{tsfasman1985group} and \cite{voloch1988note}. However, this result does not
provide immediately the statistics for cyclic EC over a given field or its
extensions, which was explored in \cite{VLADUT199913} and \cite{VLADUT1999354}. See also \cite{shparlinski2012group}, where the realizability of possible groups for EC was studied from a statistical point of view.

There are many other closely related topics which were considered in the literature. In \cite{shparlinski2005orders} and \cite{Murty2007} the order of points on the groups of rational points was studied, in particular lower bounds for the \textit{exponent} (the largest order), and some applications as well. Also, the exponent under base field extensions was examined in \cite{shparlinski2005exponent}. This question is closely related to the question about cyclicity, since cyclicity is equivalent to the fact that the exponent equals the cardinality of the group of rational points.

For both theoretical and practical interests, it is very natural to extend this investigation to abelian varieties of higher dimensions. Many facts are still unknown in this case. Nevertheless, in a recent paper \cite{Rybakov2010}, Rybakov gives a very explicit description of all possible groups of rational points of an abelian variety in a given isogeny class. His result is formulated in terms of the Newton polygon of the variety and the Hodge polygon of the group. In particular, from this description it follows that varieties with cyclic group of rational points always exist in an isogeny class.
In \cite{chantal2014surfaces} it is shown that for abelian surfaces, ``very split" groups of rational points occur with density zero. This is compatible with the general framework of the Cohen–Lenstra heuristics, which predicts that random abelian groups naturally occur with probability inversely proportional to the size of their automorphism groups. Note that the very split groups have many more automorphisms than the cyclic group of the same size.

Another related and widely explored topic is maximal curves, \ie, those with a maximal number of rational points, such curves are very important for applications. Goppa  showed that the applications of EC to coding theory are interesting only if the curve has many rational points (\cite{goppa1983algebraico}). Theoretical and practical bounds were explored and many algorithms to construct it were proposed. Relating ``maximal'' abelian varieties with its structures of rational points does not seem to be a simple task, but with many promising applications. 

In cryptography, the Discrete Logarithm Problem is extensively used in such systems and, under some conditions, cyclic groups are suitable for these purposes (\cite{GALBRAITH2005544}). In many applications, security is based on working with big cyclic subgroups of the group of rational points. 

Another interesting related topic is counting isomorphism classes of abelian varieties defined over a fixed finite field. In the $1$-dimensional case, this corresponds to compute class numbers of some orders in quadratic number fields, and in the general case, this was treated for example (very recently) in \cite{xue2017counting}. Isomorphism classes can be grouped in isogeny classes. In \cite{DIPIPPO1998426}, DiPippo and Howe developed some techniques and gave an asymptotic formula (with explicit error terms) for the number of isogeny classes of $n$-dimensional abelian varieties over the finite field with $q$ elements, for fixed $n$ and $q$ tending to infinity. 

In this paper we give some statistical results concerning cyclic isogeny classes, namely, in Theorem \ref{th:weil_polynomial_criterion} we present a simple criterion to know if an isogeny class is \emph{cyclic}, \ie, all its varieties have a cyclic group of rational points, this criterion is based on the characteristic polynomial of the isogeny class. Further, in Theorems \ref{th:weil_midterm} and \ref{th:weil_q} we give asymptotic lower bounds on the fraction of cyclic isogeny classes when some coefficients on the characteristic polynomial corresponding to these isogeny classes are fixed. Theorem \ref{th:weil_q} needs an extra condition and, we show that for ``good" coefficients, this condition holds true for abelian surfaces. Finally, we prove in Theorem \ref{th:max_surface_cyclic} that under some restrictions, isogeny classes of abelian surfaces defined over finite fields with maximal number of rational points are cyclic and ordinary. Also, we give some examples of families of cyclic isogeny classes of abelian surfaces.

The rest of this paper is organized as follows: in Section \ref{sec:PresProb} we briefly recall some general facts about abelian varieties over finite fields and we state our results more precisely; in Section \ref{sec:GenCase} we prove Theorems \ref{th:weil_polynomial_criterion} to \ref{th:weil_q} and, Section \ref{sec:AbSurf} is devoted to abelian surfaces.

\section{Preliminaries and Statement of the Results}\label{sec:PresProb}
For the general theory of abelian varieties see for example \cite{mumford1970abelian}, and for precise results over finite fields, see \cite{Waterhouse1969}.

Let $q=p^r$ be a power of a prime, and let $k=\Fq$ be a finite field with $q$ elements. Let $A$ be an abelian variety of dimension $g$ over $k$. For an extension field $K$ of $k$, we denote by $\End_K (A)$ the ring of $K$-endomorphism of $A$ and by $\Endo_K (A) =(\End_K (A))\otimes\mathbb{Q}$ its endomorphism algebra, the latter being an invariant of its isogeny class $\mathcal{A}$, we can denote it by $\Endo_K(\mathcal{A})$.
For an integer $n$, denote by $\widehat{n}$ the ratio of $n$ to the product of different prime divisors of $n$, and by $A[n]$ the group of $n$-torsion points of $A$ over $\overline{\Fq}$. Then
\begin{align}
\begin{split}\label{eqn:torsion_points}
A[n]&\cong (\Z/n\Z)^{2g}, \qquad p\nmid n\\
A[p]&\cong (\Z/p\Z)^i, \qquad 0\leq i \leq g
\end{split}
\end{align}

For a fixed prime $\ell$ ($\neq p$), the Tate module $T_\ell(A)$ is defined by $\lim\limits_{\leftarrow} A[\ell^n]$. The Frobenius endomorphism $F$ of $A$ acts on $T_\ell(A)$ by a semisimple linear operator, and its characteristic polynomial $f_{A}(t)$ is called \emph{Weil polynomial of} $A$. Tate proved in \cite{tate1966} that the isogeny class $\mathcal{A}$ of abelian variety is determined by its characteristic polynomial $f_{\mathcal{A}}(t)$ and, if $\mathcal{A}$ is simple, the center of $\Endo_k(\mathcal{A})$ is isomorphic to the number field $\mathbb{Q}(F)\cong \mathbb{Q}[t]/ (f_{\mathcal{A}}(t))$.
The cardinality of the group of rational points $A(k)$ of $A$ equals $f_{\mathcal{A}}(1)$, and thus, it is an invariant of the isogeny class. For a prime number $\ell$, $v_\ell$ denotes the usual $\ell$-adic valuation. For an abelian group $G$ we denote by $G_\ell$ the $\ell$-primary component of $G$. We denote by $\mathcal{P}$ the set of prime integers and by $\mathcal{P}(n)$ the set of prime divisors of $n$.

For an elliptic curve $E$, we have that $E(\Fq)$ is non-cyclic if and only if there exist a prime $\ell\neq p$ such that $E[\ell]\subseteq E(\Fq)$. This is not the case for higher dimensional varieties, however,  a slightly different criterion holds:
\begin{Lemma}\label{lemma:endom_criterion}
\begin{align*}\label{EndCriterion}
A(\Fq) \text{ is not cyclic } \Longleftrightarrow \exists \varphi\in \End_{\overline{k}}(A), \; N/\ell=\varphi\circ (1-F) \text{ for some prime divisor } \ell|N
\end{align*}
\end{Lemma}

Note that in this case we do not require $\ell\neq p$. Note also that if $A$ is simple, if such $\varphi$ exists, it has to be defined over $k$ and it belongs to the center of its endomorphism ring. This lemma is an easy consequence of the theory of abelian varieties over finite fields and we will prove it here within the proof of Theorem \ref{th:weil_polynomial_criterion}. Lemma \ref{lemma:endom_criterion} implies that abelian varieties with few endomorphism are more likely to be cyclic. The group structure of an abelian variety is not necessarily determined by its endomorphism ring, but from Lemma \ref{lemma:endom_criterion} it follows that the property of being cyclic or not depends on its endomorphism ring, and thus, it is an invariant of it.

We say that $A$ is \emph{cyclic} if the group $A(\Fq)$ is cyclic. 
 Counting cyclic varieties in a given isogeny class is the same as counting some fractional ideals in the center of $\Endo_k(A)$. In this paper we will focus on \emph{cyclic} isogeny classes, \ie, those that contain only cyclic varieties.
 We have the following criterion:
\begin{Theorem}\label{th:weil_polynomial_criterion}
Let $\mathcal{A}$ be a $g$-dimensional $\Fq$-isogeny class of abelian varieties corresponding to the Weil polynomial $f_\mathcal{A}(t)$. Then $\mathcal{A}$ is cyclic if and only if  $f'_\mathcal{A}(1)$ is coprime with $\widehat{f_\mathcal{A}(1)}$.
\end{Theorem}
which can be proven by using Rybakov's Theorem (\cite{Rybakov2010}) when $\Endo(A)$ is commutative, and we shall prove it using the Lemma \ref{lemma:endom_criterion} in the general case. Note that we do not require $\mathcal{A}$ to be simple.

Every Weil polynomial has the form
\begin{equation}
f_\mathcal{A}(t)=t^{2g}+a_1 t^{2g-1}+\dots + a_g t^g + a_{g-1} q t^{g-1}+\dots a_1 q^{g-1} t + q^{g} \in \mathbb{Z}[t],
\end{equation}
and for simplicity, we write it as $f_{q,(a_1,\dots,a_{g})}(t)$. 

Denote by $\mathcal{S}^g(q)$ the set of $g$-dimensional isogeny classes $\mathcal{A}$ of abelian varieties  defined over the finite field $\Fq$, and by $\mathcal{S}^g_{\mathfrak{c}}(q)$ the subset of cyclic isogeny classes. We write $(a_1,\dots,a_{g})\in \mathcal{S}^g(q)$ if the polynomial $f_{q,(a_1,\dots,a_{g})}(t)$ defines such an isogeny class. 

Our first approach is to fix a vector $\mathbf{a}=(a_1,\dots,a_{g-1})\in\Z^{g-1}$ of $g-1$ integers, and define two sets \emph{of possible values of} $a_g$:
\begin{align}
I_{\mathbf{a}}(q)&:=\{z\in\Z |  (a_1,\dots,a_{g-1},z) \in \mathcal{S}^g(q)\},\\
I_{\mathbf{a},\mathfrak{c}}(q)&:=\{z\in\Z | (a_1,\dots,a_{g-1},z) \in \mathcal{S}^g_{\mathfrak{c}}(q)\}.
\end{align}
It is obvious that $I_{\mathbf{a}}(q)$ is finite and non empty from a certain value of $q=q(\mathbf{a})$. Fixing a prime $p$, we are interested on the proportion
\begin{align}
r_{p,\mathbf{a}}(n):=\frac{\sum_{i=1}^n \# I_{\mathbf{a},\mathfrak{c}}(p^i)}{\sum_{i=1}^n \# I_{\mathbf{a}}(p^i)},
\end{align}
when $n$ tends to infinity.

Finally, if we define the polynomial $h(X,\mathbf{a}):=gf_{X,(\mathbf{a},z)}(1)-f'_{X,(\mathbf{a},z)}(1)$ (in fact, it is independent of $z$), and we have
\begin{Theorem}\label{th:weil_midterm}
Given $\mathbf{a}\in\Z^{g-1}$ and a prime $p$, then
\begin{align}
\liminf r_{p,\mathbf{a}} &\geq 1 - \frac{p}{p-1} \left[ \left(1- \prod\limits_{\ell\in \mathcal{P}(h(p,\mathbf{a}))}\left( 1-\frac{1}{\ell^2} \right) \right) \left( 1-\frac{1}{p^{g/2}}\right) +\frac{1}{p^{g/2}} \right].
\end{align}
\end{Theorem}
\begin{Rmk}
It is easy to see that if $p\gg 0$, this bound is approximately
\[
\prod\limits_{\ell\in \mathcal{P}(h(p,\mathbf{a}))}\left( 1-\frac{1}{\ell^2} \right),
\]
which is greater than $1-P(2)=0.54\dots$, where $P(s)$ is the \emph{prime zeta function}, defined as
\begin{align*}
P(s)=\sum_{\ell\in\mathcal{P}} \frac{1}{\ell^s}, \qquad s\in\mathbb{C}.
\end{align*}
\end{Rmk}

Let us follow now a different approach by fixing a vector $\mathbf{b}=(b_1,\dots,b_g)\in\Z^{g}$ of $g$ integers, we want to study when $f_{q,\mathbf{b}}(t)$ defines a cyclic isogeny class over $\Fq$ when $q$ varies over the power of primes. We consider this problem in two ways: (i)$q$ varies over primes and (ii) we fix a prime $p$ and $q$ varies over the powers of $p$.

For our next result, we will need an extra hypothesis, which we make explicit as follows: 
\begin{quote}
\textbf{Hyp}($\mathbf{b}$):
there exist integers $\eta(\mathbf{b}), t(\mathbf{b})$ and $s(\mathbf{b})$ with no common factor, such that  the only possible divisors of $f_{q, \mathbf{b}}(1)$ and $f'_{q, \mathbf{b}}(1)$ are divisors of $\eta$ and $tq+s$.
\end{quote}

\begin{Rmk}
The key idea is that these values, in particular $\eta$, are independent of $q$. Such $\eta$ can always be obtained by successive Euclidean division, since we can see $f$ and $f'$ as polynomials in $\mathbb{Q}[q]$; but it is not sure to have a linear polynomial in $q$, with the condition of no common factors. For the $2$-dimensional case, we will see in Section \ref{sec:AbSurf} that this hypothesis holds for a wide range of values of $\mathbf{b}$, and we will explicit them.
\end{Rmk}

Then we are interested on the following proportions: 
\begin{gather}
x_{\mathbf{b}}(n)=\frac{\#\{\ell\leq n \text{ prime}|  \mathbf{b}\in \mathcal{S}^g_{\mathfrak{c}}(\ell) \}}{\#\{\ell\leq n \text{ prime}| \mathbf{b}\in \mathcal{S}^g(\ell) \}},\\
y_{p,\mathbf{b}}(n)=\frac{\#\{i\leq n| \mathbf{b} \in \mathcal{S}^g_{\mathfrak{c}}(p^i)\}}{\#\{i\leq n| \mathbf{b}\in \mathcal{S}^g(p^i) \}},
\end{gather}
when $n$ tends to infinity. Note that the above fractions are defined from a big enough $n$. Then our next result is
\begin{Theorem}\label{th:weil_q}
Given $\mathbf{b}_1, \mathbf{b}_2\in\Z^g$ and a prime number $p>2$ such that
\begin{enumerate}
\item {\normalfont\textbf{Hyp}($\mathbf{b}_i$)} holds, with $\eta, t$ and $s$ given,
\item $p$ doesn't divide the last coordinate of $\mathbf{b}_2$, and
\item the multiplicative group $(\Z/\eta(\mathbf{b}_2)\Z)^*$ is generated by $p$.
\end{enumerate}
Then
\begin{align*}
1.\quad\liminf x_{\mathbf{b}_1}(n) \geq L_1 \qquad\text{ and, }\qquad 2.\quad\liminf y_{p,\mathbf{b}_2}(n) \geq L_2,
\end{align*}
where
\begin{align*}
L_i=\prod\limits_{\ell\in \mathcal{P}(\eta(\mathbf{b}_i))\setminus \mathcal{P}(t(\mathbf{b}_i) s(\mathbf{b}_i))}\frac{\ell-2}{\ell-1}, \qquad i=1,2.
\end{align*}
\end{Theorem}

\begin{Rmk}\label{rmk:optimHyp}
As we can see, this bound depends on the choose of $(\eta,t,s)$. We can verify that if $(\eta_1,t_1,s_1)$ and $(\eta_2,t_2,s_2)$ verify \textbf{Hyp}($\mathbf{b}$) for some fixed $\mathbf{b}$, then $((\eta_1,\eta_2),t_1,s_1)$ also does. Also, if $(\eta,t_1,s_1)$ and $(\eta,t_2,s_2)$ verify \textbf{Hyp}($\mathbf{b}$), and if there exist a prime $\ell\in \mathcal{P}(t_2 s_2)\setminus \mathcal{P}(t_1 s_1)$ dividing $\eta$, then $(\eta/\ell,t\ell,s\ell)$ also verifies \textbf{Hyp}($\mathbf{b}$). This allows us to optimize the bound in that way. However, in Section \ref{sec:AbSurf} we will see with an example that Theorem \ref{th:weil_q} does not provide necessarily the best bound.
\end{Rmk}

\begin{Ex}
As example, we discuss the case of elliptic curves. Here we have 
\begin{align*}
f_{q,a}(1)=(1+a)+q \qquad &\text{and} \qquad f'_{q,a}(1)=2+a, 
\end{align*}
thus, \textbf{Hyp}($a$) holds for all values of $a$ if we take
\[ 
\left \{
  \begin{array}{l}
\eta(a)=2+a, \\
t(a)=1, \\
s(a)=1+a.
  \end{array}
\right.
\]

Also, since $\eta(a)$ and $s(a)$ are always coprimes, the bound for Theorem \ref{th:weil_q} takes the form
\begin{align*}
L=\prod_{\ell\in \mathcal{P}(2+a)} \frac{\ell-2}{\ell-1}.
\end{align*}
Concerning Theorem \ref{th:weil_midterm}, we have that $h(p,\mathbf{a})=p-1$, where $\mathbf{a}$ can be considered empty. This provides a complete answer to the fraction of cyclic isogeny classes defined over some finite extension of $\Fp$, and for $p\gg 0$, the bound takes the form
\begin{align*}
\prod\limits_{\ell\in \mathcal{P}(p-1)}\left( 1-\frac{1}{\ell^2} \right).
\end{align*}
To finish this example, let us say that the question 
\begin{center}\textit{
For which values of $q$, all isomorphism classes of EC defined over $\Fq$ are cyclic?}
\end{center}
is answered in Theorem $4.1$ in \cite{VLADUT199913}, and in view of Theorem \ref{th:weil_polynomial_criterion}, this  is equivalent to
\begin{center}\textit{
For which values of $q$, $(\widehat{f_{q,a}(1)},f'_{q,a}(1))=1 \quad\forall a\in [-2\sqrt{q}, 2\sqrt{q}]$?}
\end{center}
\end{Ex} 

\section{General case: Proof of Theorems 1-3}\label{sec:GenCase}
In this section we prove Theorems \ref{th:weil_polynomial_criterion} to \ref{th:weil_q}.

\begin{proof}[Proof of Theorem \ref{th:weil_polynomial_criterion}:]
$ $\newline
First we consider the case when $\mathcal{A}$ is simple. For an abelian variety $A\in\mathcal{A}$, the center of its endomorphism algebra $\End_{\Fq}(A)\otimes\mathbb{Q}$ is isomorphic to the number field $K=\mathbb{Q}(\pi)$, where $\pi$ is a root of $f_\mathcal{A}$ and represents the Frobenius endomorphism $F$. For every variety $A$ in $\mathcal{A}$, $\End_{\Fq}(A)\cap K$ is an order in $K$ (thus contained in $\mathcal{O}_K$), and $\mathcal{O}_K$ is the endomorphism ring for some $B\in\mathcal{A}$ by Theorem 3.13 in \cite{Waterhouse1969}. For any $x\in K$, we have that $x\in\mathcal{O}_K$ holds if and only if its (unitary) minimal polynomial has integer coefficients, that is, is an algebraic integer.

Denote by $N$ the value $f_{\mathcal{A}}(1)$, then Lemma \ref{lemma:endom_criterion} 
\begin{align*}
A(\Fq) \text{ is not cyclic } \Longleftrightarrow \exists \varphi\in \End_{\overline{\Fq}}(A), \; N/\ell=\varphi\circ (1-F) \text{ for some prime divisor } \ell|N
\end{align*}
follows from the fact that $A(\Fq)$ is not cyclic if and only if $A(\Fq)=\ker (1-F) \subset \ker (N/\ell) =A[N/\ell]$ for some $\ell|N$, since $1-F$ is separable. Since $A$ is simple, if such $\varphi$ exist, it must be in $K$: if $\psi\in \End_{\Fq}(A)$, multiplying the previous equation in both sides and subtracting them, we have $0=(\varphi\psi -\psi\varphi)\circ (1-F)$, thus, $\varphi$ commutes with $\psi$. Similarly, we can show that this endomorphism is defined over $\Fq$.

We consider the minimal polynomial $p_\ell (t)$ of $\frac{N}{\ell(1-\pi)}$. All non-leading coefficients of $p_\ell$ are of the form $\frac{N^{i-1}}{\ell^i}$ times an integer, except for the coefficient of $t^{2g-1}$, which equals $-\frac{f'(1)}{\ell}$, and that the constant term equals $\frac{N^{2g-1}}{\ell^{2g}}$.

If $\mathcal{A}$ is cyclic, then all elements of $\{\frac{N}{\ell(1-\pi)}, \ell\vert N \}$ are algebraic non integers: if not, for some $\ell$, $\frac{N}{\ell(1-\pi)}\in\mathcal{O}_K$ and for some abelian variety $B$ with $\End_{\Fq}(B)\cap K=\mathcal{O}_K$ we will have that $B(\Fq)$ is not cyclic. Therefore $p_\ell$ have some non integer coefficient, either $\ell\nmid f'(1)$, either $v_\ell(N)=1$. In consequence, $gcd(\widehat{N},f'(1))=1$.

If $gcd(\widehat{N},f'(1))=1$, then for any $\ell|N$, either $(\ell,f'(1))=1$, or $v_\ell(N)=1$. In both cases, $p_\ell$ have some non integer coefficient, and thus, $\frac{N}{\ell(1-\pi)}\notin \mathcal{O}_K$. This implies that, for all $A\in\mathcal{A}$, there is no $\varphi\in \End(A)$ with $\frac{N}{\ell}=\varphi\circ (1-F)$: if such $\varphi$ exist, it has to be equal to $\frac{N}{\ell(1-\pi)}$, which is a contradiction.

Now we consider the general case. Then $\mathcal{A}$ is isogenous to a product $\prod \mathcal{A}_i$ of simple isogeny classes, and we have that $f_\mathcal{A} (t)=\prod f_i(t)$, where $f_i(t)$ is the characteristic polynomial of $\mathcal{A}_i$. Suppose $\mathcal{A}$ is not cyclic, we have two possibilities: (i) one of its components is not cyclic or (ii) $\exists \ell$ prime such that $\ell| (f_1(1),f_2(1))$ (WLOG). In (i), by the first part of the theorem, $(\widehat{f_1(1)},f'_1(1))>1$ (WLOG), writing $f'_\mathcal{A} (t)= f'_1(t) \prod_{i>1} f_i(t)+f_1(t) (\prod_{i>1} f_i(t))'$ we have that $(\widehat{ f_\mathcal{A}(1)}, f'_\mathcal{A}(1))>1$. In (ii), writing $f'_\mathcal{A} (t)= f'_1(t) f_2(t) \prod_{i>2} f_i(t) + f_1(t) f'_2(t) \prod_{i>2} f_i(t) + f_1(t) f_2(t) (\prod_{i>2} f_i(t))'$ we have that $(\widehat{ f_\mathcal{A}(1)},f'_\mathcal{A}(1))>1$. Suppose now that $\ell | (\widehat{f_{\mathcal{A}}(1)}, f'_{\mathcal{A}}(1))$, thus, $\ell^2|f_{\mathcal{A}}(1)$ and $\ell | f'_{\mathcal{A}}(1)$. Again, we have two possibilities: (i) $\ell| f_1(1)$ and $\ell| \prod_{i>1} f_i(1)$ or, (ii) $\ell^2| f_1(1)$ and $\ell\nmid \prod_{i>1} f_i(1)$. In (i), $\mathcal{A}$ is not cyclic. In (ii), writing $f'_\mathcal{A} (t)= f'_1(t) \prod_{i>1} f_i(t)+f_1(t) (\prod_{i>1} f_i(t))'$ we see that $\ell |f'_1(1)$, thus, $\mathcal{A}$ is not cyclic by the first part of the theorem. Theorem \ref{th:weil_polynomial_criterion} is thus proved.
\end{proof}

Note that Lemma 3.3.1 in \cite{DIPIPPO1998426} implies that for fixed $\mathbf{a}\in\Z^{g-1}$, $f_{q,(\mathbf{a},z)}(t)$ always defines an ordinary isogeny class from a $q$ big enough when $p\nmid z$, and they are almost all the isogeny classes. This will be important for the proofs of both theorems that follow. 

\begin{proof}[Proof of Theorem \ref{th:weil_midterm}:]
$ $\newline
From now on, $\mathbf{a}$ and $p$ are fixed. For simplicity, we denote by $N_i$ and $S_i$ the integers $\# I_{\mathbf{a}}(p^i)$ and the sets of prime divisors of $h(p^i, \mathbf{a})$, with $S=S_1$, respectively. We set $M_i=\max(I_{\mathbf{a}}(q))-\min(I_{\mathbf{a}}(q))$. Also, we write $f_{i,z}(t)$ instead of $f_{p^i,(\mathbf{a},z)}(t)$.

From Lemma 3.3.1 in \cite{DIPIPPO1998426} we deduce that 
\begin{align*}
N_i/M_i \rightarrow \frac{p-1}{p} \qquad \text{ and }\qquad M_i/4p^{ic}\rightarrow 1 \qquad \text{ when } i\rightarrow \infty,
\end{align*}
with $c=g/2$.

The following elementary lemma allows us to identify the primes that can simultaneously divide the arithmetic progressions $f_{i,z}(1)$ and $f'_{i,z}(1)$ (as sequences on $z$): 

\begin{Lemma}
Let $x_n$ and $y_n$ be two arithmetic progressions of difference $d_x$ and $d_y$ respectively, and let $\ell$ be a prime such that $d_x \not\equiv 0 \pmod{\ell}$. Then $\ell|(x_i,y_i)$ for some $i$ if and only if $d_x y_j \equiv d_y x_j \pmod{\ell}$ for some $j$.
\end{Lemma}

\begin{proof}
First, observe that if the congruence $d_x y_j \equiv d_y x_j \pmod{\ell}$ is verified for some $j$, then it verifies for all $i$:
\begin{align*}
d_x y_i &\equiv d_x(y_i - y_j + y_j) \pmod{\ell} \\
&\equiv d_x y_j + d_x d_y (i-j) \\
&\equiv d_y x_j + d_x d_y (i-j) \\
&\equiv d_y(x_j + d_x (i-j)) \\
&\equiv d_y x_i
\end{align*}
Suppose now that $d_x y_j \equiv d_y x_j 
\pmod{\ell}$, then we have to solve the system

\[ 
\left \{
  \begin{array}{r}
x_0 + d_x i \equiv 0  \pmod{\ell} \\
y_0 + d_y i \equiv 0  \pmod{\ell} 
  \end{array}
\right.
\]
We get that $i\equiv -x_0 d_x^{-1}$, and thus, $y_0 + d_y (-x_0 d_x^{-1}) \equiv y_0 -  (d_y x_0) d_x^{-1} \equiv y_0 - d_x y_0 d_x^{-1} \equiv 0 \pmod{\ell}$
\end{proof}

We are in the case where $g(f_{i,z}(1)-z)\equiv f'_{i,z}(1)-gz \pmod{\ell}$, thus, only the primes in $S_i$ can simultaneously divide $f_{i,z}(1)$ and $f'_{i,z}(1)$. 

For an integer $n$ and a finite set of integers $Z$, let denote 
\begin{align*}
\sigma(n,Z)&=\sum\limits_{i\geq 1} (-1)^{i+1}\sum\limits_{R\subset Z, |R|=i} \left\lceil\frac{n}{\prod_{x\in R} x^2}\right\rceil, \\
\xi(Z)&=\sum\limits_{i\geq 1} (-1)^{i+1}\sum\limits_{R\subset Z, |R|=i} \frac{1}{\prod_{x\in R} x^2}=1-\prod\limits_{x\in Z}\left( 1-\frac{1}{x^2} \right).
\end{align*}
Observe that $\xi$ can be defined for infinite subsets of $\mathcal{P}$ since $\xi(\mathcal{P})$ converges and, if $Z\subseteq Z'$, $\xi(Z) \leq \xi(Z')$. 

We will study the sequence $1-r_{p,\mathbf{a}}$, \ie, we will count non cyclic isogeny classes using the lemma above. Note that the number of non cyclic isogeny classes for $p^i$ is at most $\sigma(M_i,S_i)$.\\

For every finite set $T$ of prime numbers that doesn't divides $h(p, \mathbf{a})$ ($T\subset\mathcal{P}\setminus S_1$ finite), 
we consider the polynomial 
\begin{align*}
g(X)=m(g)X+1, \quad m(g)=\left(\prod\limits_{\ell\in T}\ell-1\right),
\end{align*}
thus, $h(p^i, \mathbf{a}) \equiv h(p, \mathbf{a}) \not\equiv 0 \pmod{\ell}$ for all $\ell\in T$ and for all $i\in g(\mathbb{N})$, hence $S_i\subset \mathcal{P}\setminus T$ for all $i\in g(\mathbb{N})$.

We get
\begin{align*}
1-r_{p,\mathbf{a}}(n) &\leq \frac{\sum \sigma(M_i,S_i)}{\sum N_i} \leq \frac{\sum\limits_{i\notin g(\mathbb{N})} M_i+\sum\limits_{i\in g(\mathbb{N})} \sigma(M_i,S_i)}{\sum N_i} \\
&\approx \frac{p}{p-1}\frac{\sum\limits_{i\notin g(\mathbb{N})} M_i+\sum\limits_{i\in g(\mathbb{N})} \sigma(M_i,S_i)}{\sum M_i} \\
&\approx \frac{p}{p-1}\frac{\sum\limits_{i\notin g(\mathbb{N})} M_i+\sum\limits_{i\in g(\mathbb{N})} M_i\xi(S_i)}{\sum M_i} \\
&\leq\frac{p}{p-1}\frac{\sum\limits_{i\notin g(\mathbb{N})} M_i+\sum\limits_{i\in g(\mathbb{N})} M_i\xi(\mathcal{P}\setminus T)}{\sum M_i} \\
&\approx \frac{p}{p-1}\frac{\sum\limits_{i\notin g(\mathbb{N})} (p^{c})^i+\xi(\mathcal{P}\setminus T)\sum\limits_{i\in g(\mathbb{N})} (p^{c})^{i}}{\sum (p^{c})^i} \\
\end{align*}
where all the sums are taken over $1\leq i\leq n$ and $\approx$ means that both sequences have the same limit. The last sequence converges to
\begin{align*}
\frac{p}{p-1} \left[1+ (\xi(\mathcal{P}\setminus T)-1)\frac{p^{cm(g)}(p^{c}-1)}{p^{c(m(g)+1)}-1}\right].
\end{align*}
As we have the previous reasoning for all finite set $T$ disjoint to $S$, thus,
\begin{align*}
\limsup (1-r_{p,\mathbf{a}}(n)) \leq \frac{p}{p-1}\left[1+(\xi(S)-1)(1-\frac{1}{p^c})\right]=\frac{p}{p-1} \left[ \xi(S) \left( 1-\frac{1}{p^c}\right) +\frac{1}{p^c} \right].
\end{align*}
Finally
\begin{align*}
\liminf r_{p,\mathbf{a}} \geq 1 - \frac{p}{p-1} \left[ \xi(S) \left( 1-\frac{1}{p^c}\right) +\frac{1}{p^c} \right].
\end{align*}
That finish the proof of Theorem \ref{th:weil_midterm}.
\end{proof}

\begin{proof}[Proof of Theorem \ref{th:weil_q}:]
$ $\newline
We write $q$ for $\ell$ varying (i) over the primes or (ii) over powers $p^i$ when $i$ varies over positive integers.

As a consequence of Dirichlet's Theorem or Theorem's hypothesis, the values of $q$ are (asymptotically) evenly distributed among the invertible congruence classes modulo $\eta$, when $q$ varies according to (i) or (ii), respectively.

Set $\mathfrak{j}(q)=tq+s$, then, by \textbf{Hyp}, the condition $(\mathfrak{j},\eta)=1$ would imply that $\mathbf{b}_i \in \mathcal{S}_{\mathfrak{c}}^g(q)$, from a certain value of $q$ when $f_{q,\mathbf{b}_i}(t)$ defines an isogeny class. 

Observe that $(\mathfrak{j},\eta)=1$ if and only if $\mathfrak{j}$ is invertible in $\Z/\eta\Z$; then we conclude by using the following lemma.
\end{proof}

\begin{Lemma}\label{lemma_jdistribution}
Let $t,s$ and $n$ be integers with no common factor, then
\begin{align*}
\frac{\# \{ x\in (\Z/n\Z)^*, tx+s\in (\Z/n\Z)^*\} }{\# (\Z/n\Z)^*} =\prod\limits_{\ell\in \mathcal{P}(n)\setminus \mathcal{P}(ts)} \frac{\ell-2}{\ell-1}
\end{align*}
\end{Lemma}
\begin{proof}
Take an $\ell\in \mathcal{P}(n)$ and set $m=\ell^{v_\ell(n)}$, we have that $\# (\Z/m\Z)^*=m-m/\ell$. If $\ell\nmid ts$ we have that $\# \{ x\in (\Z/m\Z)^*, tx+s\in (\Z/m\Z)^*\}=m-2m/\ell$. If $\ell|ts$, for all $x\in(\Z/m\Z)^*$ we have that $tx+s\in (\Z/m\Z)^*$. We conclude by using the isomorphism $\Z/n\Z\cong\prod_{\ell|n}\Z/\ell^{v_\ell(n)}\Z$ 
\end{proof}

\section{The case of Abelian Surfaces}\label{sec:AbSurf}
In this section we consider the case of abelian surfaces, where the characteristic polynomial takes the form 
\begin{align*}
f_\mathcal{A}(t)=t^{4}+at^3 + b t^2 + a q t + q^2,
\end{align*}
thus, we have $\mathbf{b}=(a,b)$. 

\subsection{Theorems applied to Surfaces.} In order to apply Theorem \ref{th:weil_q} to surfaces, we need to verify \textbf{Hyp}($\mathbf{b}$); and this is a consequence of the proposition below and the conditions on $\mathbf{b}$ that follow:
\begin{Prop}\label{lemma_eta}
Let $a, b$ and $q$ be integers, then the common divisors of
\begin{align*}
N&=1+a(q+1)+b+q^2 \text{ and }\\
J&=4+a(q+3)+2b
\end{align*}
are divisors of $\eta(a,b):=(2a+b+2)(a^2-4a+4b-16)$.
\end{Prop}
\begin{proof}

Let $\ell$ be a prime divisor of $N$ and $J$, then it divides
\begin{align*}
(q+3)N-(q+1)J&=(q-1)(q^2+4q-1+b) \text{ and}\\
2N-J&=(q-1)(2q+2+a)
\end{align*}
If $\ell$ divides $q-1$, then $N=q^2-1 + a(q-1) +2a +b +2 \equiv 2a+b+2 (\mod \ell)$.
We set 
\begin{align*}
\gamma &=q^2+4q-1+b\\
\theta &=2q+2+a
\end{align*}
So, if $\ell$ doesn't divide $q-1$, then it divides $\gamma$ and $\theta$, and
\begin{align*}
4\gamma=\theta(2q+6-a)+(a^2-4a+4b-16)
\end{align*}
Finally, the possible common divisors are the divisors of $(2a+b+2)(a^2-4a+4b-16)$
\end{proof}
We have $\eta(a,b)$ as defined in the proposition and clearly $f'(1)=aq+4+3a+2b$, thus, we can take 
\[
\left\{
\begin{array}{l}
\eta(a,b)=(2a+b+2)(a^2-4a+4b-16),\\
t(a,b)=a,\\
s(a,b)=4+3a+2b.
\end{array}\right.
\]
To finish, and in order to have the condition of no common divisor (with the values previously chosen), the integers in $\mathbf{b}=(a,b)$ must verify:
\begin{align*}
(a,b+2)=1, \qquad\text{ and }\qquad a\equiv 1 \pmod{2}. 
\end{align*}

A priori these values are not optimal in the sense of Remark \ref{rmk:optimHyp}, and as mentioned previously, Theorem \ref{th:weil_q} does not necessarily provide the optimal bound. For example, observe that  \textbf{Hyp}($0,0$) holds with $\eta=2, t=1,s=1$, in which case $L_1=L_2=0$, which are clearly very bad bounds. But by applying ``manually'' Theorem \ref{th:weil_polynomial_criterion} we can have better bounds, first we have: 
\begin{align*}
f_{q,\mathbf{b}}(1)=1+q^2=(q^2-1)+2 \qquad \text { and } \qquad f'_{q,\mathbf{b}}(1)=4,
\end{align*}
thus, $f_{q,\mathbf{b}}(1)$ is equivalent to $1$ or $2$ modulo $4$, depending if $p$ equals $2$ or not, respectively, thus $L_1=L_2=1$. This is because Theorem \ref{th:weil_q} uses only a weak result:
\begin{align*}
(f(1),f'(1))=1 \qquad\Rightarrow \qquad \text{cyclicity}, 
\end{align*}
and not the stronger equivalence from Theorem \ref{th:weil_polynomial_criterion}:
\begin{align*}
(\widehat{f(1)},f'(1))=1 \qquad\Leftrightarrow \qquad \text{cyclicity}. 
\end{align*}

Concerning Theorem \ref{th:weil_midterm} applied to abelian surfaces, this provides the fraction of cyclic isogeny classes with a fixed trace of Frobenius $a$ and a fixed base field characteristic $p$. We have that $h(p,a)=(p-1)[2(p+1)+a]$. Thus, for $a=-4$ and a Fermat prime $p$, the lower bound for Theorem \ref{th:weil_midterm} is $0.75$ and we cannot improve it. Even for primes of form $2\ell+1$ (with $\ell\gg 0$ a prime number) we have a good bound, around $0.75$.

\subsection{Maximal Surfaces}\label{ssec:MaxSurf}
Now, let us discuss about the other topic mentioned in the Introduction, varieties with many rational points. R\"uck's theorem \cite{Ruck1990} describes polynomials that occurs as Weil polynomials for abelian surfaces with endomorphism algebra being a field:

\begin{Theorem}(R\"uck)
The set of $f_A(t)$ for all abelian varieties $A$ over $\Fq$ of dimension 2 whose algebra $\End(A)\otimes \mathbb{Q}$ is a field is equal to the set of polynomials $f(t)=t^4+at^3+bt^2+aqt+q^2$ where the integers $a$ and $b$ satisfy the conditions
\begin{enumerate}
\item $\vert a \vert < 4\sqrt{q}, \quad 2\vert a\vert\sqrt{q}-2q < b < a^2/4+2q$
\item $\Delta=a^2-4b+8q$ is not a square in $\mathbb{Z}$, and
\item either
\begin{enumerate}
\item $v_p(a)=0, \quad v_p(b)\geq r/2$ and $(b+2q)^2-4qa^2$ is not a square in $\mathbb{Z}_p$,
\item\label{ordinarycase} $v_p(b)=0$, or
\item $v_p(a)\geq r/2, v_p(b) \geq r$, and $f(t)$ has no root in $\mathbb{Z}_p$.
\end{enumerate}
\end{enumerate}
\end{Theorem}

Concerning this type of isogeny classes, we have the following result

\begin{Theorem}\label{th:max_surface_cyclic}
Let $q=p^r$ with $r$ even. Consider the $\Fq$-isogeny classes $\mathcal{A}$ of abelian surfaces  
with $\Endo_{\Fq}(\mathcal{A})$ being a field. Then the isogeny class with maximal number of rational points among such classes is ordinary and cyclic, and corresponds to the Weil polynomial
\begin{align*}
f(t)=t^4+at^3+bt^2+aqt+q^2,\quad \text{where} \quad a=4\sqrt{q}-3 \quad \text{and} \quad b=6q-6\sqrt{q}+1.
\end{align*}
\end{Theorem}
\begin{proof}
For a fixed $q$, the number of rational points increases with $a$ and $b$. From the conditions of R\"uck's Theorem we have that the biggest possible values are $a=4\sqrt{q}-3$ and $b=6q-6\sqrt{q}+1$ with $\Delta=5$, and the ordinariness is clear.

Then we have 
\begin{align*}
f(1)&=-1-2\sqrt{q}+3q+4(\sqrt{q})^3+q^2 \text{ and}\\
f'(1)&=-3+9q+4(\sqrt{q})^3
\end{align*}
Writing $f(1)$ and $f'(1)$ as functions $N$ and $\mathfrak{j}$ of $x=\sqrt{q}$, respectively, we see that
\begin{align*}
N(x)=x^4+4x^3+3x^2-2x-1,\qquad \mathfrak{j}(x)=4x^3+9x^2-3,
\end{align*}
and thus,
\begin{align*}
(14x^3+49x^2+28x-14)\mathfrak{j}(x)-(56x^2+98x-7)N(x)=35=5\cdot 7
\end{align*}

Finally, we can see case by case, that for $\ell\in \{5,7\}$ we cannot have that $\ell|\mathfrak{j}(x)$ and $\ell^2|N(x)$ simultaneously:
\begin{itemize}
\item for $\ell=5$, 
\begin{center}
  \begin{tabular}{ | c | c |c |c |c |c |}
    \hline
    $x \pmod{5}$ & 0 & 1 & 2 & 3 & 4 \\ \hline
    $\mathfrak{j} \pmod{5}$ & 2 & 0 & 0 & 1 & 2 \\ \hline
  \end{tabular}
    \begin{tabular}{ | c | c |c |c |c |c |c |c |c |c |c |}
    \hline
    $x \pmod{25}$ & 1 & 2 & 6 & 7 & 11 & 12 & 16 & 17 & 21 & 22 \\ \hline
    $N \pmod{25}$ & 5 & 5 & 5 & 5 & 5 & 5 & 5 & 5 & 5 & 5 \\ \hline
  \end{tabular}
\end{center}
\item for $\ell=7$, 
\begin{center}
  \begin{tabular}{ | c | c |c |c |c |c |c |c |}
    \hline
    $x \pmod{7}$ & 0 & 1 & 2 & 3 & 4 & 5 & 6 \\ \hline
    $\mathfrak{j} \pmod{7}$ & 4 & 5 & 3 & 1 & 2 & 2 & 4 \\ \hline
  \end{tabular}
  \end{center}
\end{itemize}
\end{proof}

Surfaces having more number of rational points $N$ than those considered in Theorem \ref{th:max_surface_cyclic} are only (in decreasing order of $N$): 
\[
E_{max}^2, \quad E_{max}\times E_{max-1}, \quad E_{max}\times E_{max-2}, \quad E_{max-1}^2 \quad \text{and} \quad E_{max-1}\times E_{max-2},
\]
 where $E_{max}, E_{max-1}$ and $E_{max-2}$ are the elliptic curves defined by its Frobenius traces $2\sqrt{q}, 2\sqrt{q}-1$ and $2\sqrt{q}-2$, respectively. We can easily check, by using Theorem \ref{th:weil_polynomial_criterion}, that $E_{max}$ is always not cyclic, $E_{max-2}$ is always cyclic, and $E_{max-1}$ is cyclic or not, depending on $q\equiv 0,1 \pmod{3}$ or $q\equiv 2 \pmod{3}$, respectively (since $3$ is the only possible prime divisor of $f(1)$ and $f'(1)$); thus, the cyclicity of $E_{max-1}\times E_{max-2}$ depends only on the cyclicity of $E_{max-1}$, since the cardinalities of $E_{max-1}$ and $E_{max-2}$ are relatively prime.

On the other hand, note that for $q=23^3$ the ordinary class with maximal number of rational points, among the classes considered in Theorem \ref{th:max_surface_cyclic}, is given by 
\begin{align*}
f(t)=t^4+438t^3+72293t^2+438(23^3)t+23^6
\end{align*}
and is not cyclic since $7|f'(1)$ and $49|f(1)$, so Theorem \ref{th:max_surface_cyclic} holds only for even powers. 

\subsection{Families of cyclic isogeny classes}
We finish this section by giving two sequences of cyclic isogeny classes of abelian surfaces in Proposition \ref{prop:seq_cyclic_1} and \ref{prop:seq_cyclic_2}. Basically, we get the cyclicity of these families by choosing appropriate values of $q$ and $\mathbf{b}$.

\begin{Prop}\label{prop:seq_cyclic_1}
Given an integer $b$ and a power of a prime $q=p^r>4$, such that: 
\begin{enumerate}
\item $p\nmid b$
\item $8\sqrt{q}-4q\leq b\leq 4$,
\item $N_0=b-3-2q+q^2$ is coprime with $N'_0=2(b-4)$
\end{enumerate}
Let be $s=\prod\limits_{\ell\in\mathcal{P}(N'_0)} (\ell-1)$.
Then for $i\in\Z_{\geq 0}$
\begin{align*}
f_i(t)=t^4-4t^3+b_it^2-4q_it+q_i^2,\quad \text{where} \quad q_i=p^{r+is} \quad \text{and} \quad b_i=b+2q_i,
\end{align*}
defines a ordinary cyclic isogeny class $\mathcal{A}_i$ of abelian surfaces over $\mathbb{F}_{q^i}$.
\end{Prop}
\begin{proof}
Note that $f_i(t)=f_{q_i,(-4,b_i)}(t)$. Every $f_i$ defines an ordinary isogeny class of abelian surfaces since $p\nmid b_i$, and $f_i$ is a Weil polynomial since
\begin{gather*}
8\sqrt{q}p^{is/2}-4qp^{is} \leq 8\sqrt{q}-4q \leq b \qquad\Rightarrow\qquad 8\sqrt{q_i}-2q_i \leq b+2q_i=b_i\\
b_i=b+2q_i \leq 4+2q_i=\frac{(-4)^2}{4}+2q_i.
\end{gather*}

Observe that $f'_{q_i,(-4,b_i)}(1) = 4-4(q_i+3)+2(b+2q_i)=N'_0=f'_{q,(-4,b)}(1)$. Also, we have that
\begin{align*}
f_{i+1}(1) - f_i(1) = (p^s-1)[p^{2(r+is)}(p^s+1)+2p^{r+is}].
\end{align*}
Thus, $(f_{i}(1),N'_0)=1$ implies $(f_{i+1}(1),N'_0)=1$, then if $\mathcal{A}_i$ is cyclic, $\mathcal{A}_{i+1}$ is cyclic as well. We conclude by using the theorem's hypothesis.
\end{proof}

It is not hard to find values $(q,b)$ verifying the last hypothesis. Moreover, for a fixed prime $p$, there are infinite many pairs $(r,b)$ such that $(p^r,b)$ verifies it.

\begin{Prop}\label{prop:seq_cyclic_2}
Given an integer $b$ and an odd prime $p>b+2$ such that $b\not\equiv 2 (\mod 4)$ and 
\[
\exists r\geq 1, (p^{2r}-1, b+2)=1,
\]
then for every $i\geq 0$,
\[
f_i(t)=t^4+bt^2+q_i^2
\]
defines an ordinary cyclic isogeny class $\mathcal{A}_i$ of abelian surfaces over $\Fq[i]{}$, where $q_i=p^{r+is}$ and  
\begin{align*}
s=\prod\limits_{\ell\in\mathcal{P}(b+2)} (\ell-1).
\end{align*}
\end{Prop}
\begin{proof}

It is easy to check that $f_i(t)$ is a Weil polynomial and it defines an ordinary isogeny class since $p>b$. 
Put
\begin{align*}
c=\prod\limits_{\ell\in\mathcal{P}(b+2)}  \ell,
\end{align*}
We have
\begin{align*}
f_i(1)=p^{2(r+is)}-1+(b+2), \; i\in\Z_{\geq 0}.
\end{align*}

Observe that $f_i(t)=f_{q_i,(0,b)}(t)$ and $f'_i(1)=2(b+2)$. For $\mathcal{A}_i$ to be cyclic it is enough to have $(f_i(1),2(b+2))=1$, and since $f_i(1)\equiv b+2 \not\equiv 0 \pmod{4}$, it is enough to have $(f_i(1),b+2)=(f_i(1),c)=1$. The last statement is true since by hypothesis $(f_0(1),c)=1$ and since $c$ divides the difference $f_{i+1}(1)-f_i(1)=p^{2(r+is)}(p^s+1)(p^s-1)$ (using Fermat's little theorem).
\end{proof}

For a fixed $p$, big enough, take a prime $\ell$ such that $2<\ell <p$ and $\ell\nmid p^2 -1$, then set $b=\ell -2$. This gives $b$ and $p$ satisfying the proposition hypothesis.

\section*{Acknowledgement}
I would like to thank my advisor Serge Vl\u adu\c t for very fruitful discussions and for his very useful comments and suggestions.

\bibliography{ajgiangreco-CAVFF}

\begin{thebibliography}{10}

\bibitem{shparlinski2012group}
{\sc W.~D. Banks, F.~Pappalardi, and I.~E. Shparlinski}, {\em On group
  structures realized by elliptic curves over arbitrary finite fields},
  Experimental Mathematics, 21 (2012), pp.~11--25.

\bibitem{chantal2014surfaces}
{\sc C.~David, D.~Garton, Z.~Scherr, A.~Shankar, E.~Smith, and L.~Thompson},
  {\em Abelian surfaces over finite fields with prescribed groups}, Bulletin of
  the London Mathematical Society, 46 (2014), pp.~779--792.

\bibitem{DIPIPPO1998426}
{\sc S.~A. DiPippo and E.~W. Howe}, {\em Real polynomials with all roots on the
  unit circle and abelian varieties over finite fields}, Journal of Number
  Theory, 73 (1998), pp.~426 -- 450.

\bibitem{GALBRAITH2005544}
{\sc S.~Galbraith and A.~Menezes}, {\em Algebraic curves and cryptography},
  Finite Fields and Their Applications, 11 (2005), pp.~544 -- 577.
\newblock Ten Year Anniversary Edition!

\bibitem{goppa1983algebraico}
{\sc V.~D. Goppa}, {\em Algebraico-geometric codes}, Izvestiya: Mathematics, 21
  (1983), pp.~75--91.

\bibitem{Kaliski1991}
{\sc B.~S. Kaliski}, {\em One-way permutations on elliptic curves}, Journal of
  Cryptology, 3 (1991), pp.~187--199.

\bibitem{koblitz1987elliptic}
{\sc N.~Koblitz}, {\em Elliptic curve cryptosystems}, Mathematics of
  computation, 48 (1987), pp.~203--209.

\bibitem{shparlinski2005exponent}
{\sc F.~Luca and I.~E. Shparlinski}, {\em On the exponent of the group of
  points on elliptic curves in extension fields}, International Mathematics
  Research Notices, 2005 (2005), pp.~1391--1409.

\bibitem{Morain1991}
{\sc F.~Morain}, {\em Building cyclic elliptic curves modulo large primes}, in
  Advances in Cryptology --- EUROCRYPT '91, D.~W. Davies, ed., Berlin,
  Heidelberg, 1991, Springer Berlin Heidelberg, pp.~328--336.

\bibitem{mumford1970abelian}
{\sc D.~Mumford}, {\em Abelian varieties}, vol.~5 of Tata Institute of
  fundamental research studies in mathematics, Oxford University Press, 1970.

\bibitem{Murty2007}
{\sc R.~Murty and I.~Shparlinski}, {\em Group Structure Of Elliptic Curves Over
  Finite Fields And Applications}, Springer Netherlands, Dordrecht, 2007,
  pp.~167--194.

\bibitem{ruck1987note}
{\sc H.-G. R\"{u}ck}, {\em A note on elliptic curves over finite fields},
  Mathematics of Computation, 49 (1987), pp.~301--304.

\bibitem{Ruck1990}
\leavevmode\vrule height 2pt depth -1.6pt width 23pt, {\em Abelian surfaces and
  jacobian varieties over finite fields}, Compositio Mathematica, 76 (1990),
  pp.~351--366.

\bibitem{Rybakov2010}
{\sc S.~Rybakov}, {\em The groups of points on abelian varieties over finite
  fields}, Central European Journal of Mathematics, 8 (2010), pp.~282--288.

\bibitem{SCHOOF1987183}
{\sc R.~Schoof}, {\em Nonsingular plane cubic curves over finite fields},
  Journal of Combinatorial Theory, Series A, 46 (1987), pp.~183 -- 211.

\bibitem{shparlinski2005orders}
{\sc I.~E. Shparlinski et~al.}, {\em Orders of points on elliptic curves},
  Contemporary Mathematics, 369 (2005), pp.~245--252.

\bibitem{tate1966}
{\sc J.~Tate}, {\em Endomorphisms of abelian varieties over finite fields},
  Invent. Math., 2 (1966), pp.~134--144.

\bibitem{tsfasman1985group}
{\sc M.~A. Tsfasman}, {\em Group of points of an elliptic curve over a finite
  field}, Theory of numbers and its applications, {Tbilisi},  (1985),
  pp.~286--287.

\bibitem{VLADUT199913}
{\sc S.~G. Vl\u{a}du\c{t}}, {\em Cyclicity statistics for elliptic curves over
  finite fields}, Finite Fields and Their Applications, 5 (1999), pp.~13 -- 25.

\bibitem{VLADUT1999354}
\leavevmode\vrule height 2pt depth -1.6pt width 23pt, {\em On the cyclicity of
  elliptic curves over finite field extensions}, Finite Fields and Their
  Applications, 5 (1999), pp.~354 -- 363.

\bibitem{voloch1988note}
{\sc J.~F. Voloch}, {\em A note on elliptic curves over finite fields}, Bull.
  Soc. Math. France, 116 (1988), pp.~455--458.

\bibitem{Waterhouse1969}
{\sc W.~C. Waterhouse}, {\em Abelian varieties over finite fields}, Annales
  scientifiques de l'\'Ecole Normale Sup\'erieure, 2 (1969), pp.~521--560.

\bibitem{xue2017counting}
{\sc J.~Xue and C.-F. Yu}, {\em Counting abelian varieties over finite fields},
  arXiv preprint arXiv:1801.00229,  (2017).

\end{thebibliography}
\bibliographystyle{siam}

\end{document}